\documentclass[12pt]{amsart}

\usepackage{amssymb}
\usepackage{bm}
\usepackage{graphicx}
\usepackage{psfrag}
\usepackage{color}
\usepackage{url}
\usepackage{algpseudocode}
\usepackage{fancyhdr}
\usepackage{xy}
\input xy
\xyoption{all}
\usepackage{stmaryrd}
\usepackage{calrsfs}

\newtheorem{theorem}{Theorem}[section]
\newtheorem{proposition}[theorem]{Proposition}

\newtheorem{lemma}[theorem]{Lemma}
\newtheorem{cor}[theorem]{Corollary}
\newtheorem{conjecture}[theorem]{Conjecture}
\theoremstyle{definition}
\newtheorem{definition}[theorem]{Definition}
\newtheorem{example}[theorem]{Example}

\numberwithin{equation}{section}

\newcommand\nn{\mathbb{N}}
\newcommand\pp{\mathbb{P}}
\newcommand\qq{\mathbb{Q}}
\newcommand\rr{\mathbb{R}}
\newcommand\zz{\mathbb{Z}}

\newcommand\mm{\mathcal{M}}
\providecommand\ldb{\llbracket}
\providecommand\rdb{\rrbracket}
\newcommand{\twopf}[4]
{
	\left\{
	\begin{array}{ll}
		#1 & \mbox{if } #2 \\
		#3 & \mbox{if } #4
	\end{array}
	\right.
}



\begin{document}
	
	\mbox{}
	\title{On the set of molecules \\ of numerical and Puiseux monoids}
	
	\author{Marly Gotti}
	\address{Department of Research and Development \\ Biogen \\ Cambridge \\ MA 02142 \\ USA}
	\email{marly.cormar@biogen.com}
	
	\author{Marcos M. Tirador}
	\address{Facultad de Matem\'atica y Computaci\'on \\ Universidad de La Habana \\ San L\'azaro y~L, Vedado \\ Habana 4 \\ CP-10400 \\ Cuba}
	\email{marcosmath44@gmail.com}
	
	\date{\today}
	
	\begin{abstract}
		Additive submonoids of $\qq_{\ge 0}$, also known as Puiseux monoids, are not unique factorization monoids (UFMs) in general. Indeed, the only unique factorization Puiseux monoids are those generated by one element. However, even if a Puiseux monoid is not a UFM, it may contain nonzero elements having exactly one factorization. We call such elements molecules. Molecules were first investigated by W. Narkiewicz in the context of algebraic number theory. More recently, F. Gotti and the first author studied molecules in the context of Puiseux monoids. Here we address some aspects related to the size of the sets of molecules of various subclasses of Puiseux monoids with different atomic behaviors. In particular, we positively answer the following recent realization conjecture: for each $m \in \nn_{\ge 2}$ there exists a numerical monoid whose set of molecules that are not atoms has cardinality $m$.
	\end{abstract}
	\medskip
	
	\maketitle
	
	\medskip
	
	
	\section{Introduction}
	\label{sec:intro}
	
	Let $M$ be a cancellative and commutative monoid. A factorization of a non-invertible element $x \in M$ is a formal product $a_1 \cdots a_\ell$ of atoms (i.e., irreducible elements), up to permutations and associates, such that $x = a_1 \cdots a_\ell$ in $M$; in this case, $\ell$ is called the length of the factorization. Following P. M. Cohn, we call $M$ atomic if every non-invertible element of $M$ has a factorization. In addition, $M$ is called  a unique factorization monoid (or a UFM) if every non-invertible element of $M$ has a unique factorization. Although each UFM is clearly atomic, an element of an atomic monoid may have more than one factorization (even infinitely many). For instance, this is the case of the element $6$ in the multiplicative monoid of the ring of algebraic integers $\zz[\sqrt{-5}]$; notice that 
	\[
	6 = 2 \cdot 3 =  (1 - \sqrt{-5})(1 + \sqrt{-5}).
	\]
	A didactic exposition of the factorization-theoretical aspects of $\zz[\sqrt{-5}]$ can be found in~\cite{CGG19}. Following \cite{GG20}, we say that a non-invertible element $x \in M$ is a molecule if $x$ has exactly one factorization in $M$, and we let $\mm(M)$ denote the set consisting of all molecules of $M$. 
	
	\smallskip
	Perhaps, the first systematic study of molecules was carried out by W. Narkiewicz back in the 1960s in multiplicative monoids of rings of algebraic integers of quadratic number fields  \cite{wN66,wN66a} and later in rings of integers of general number fields~\cite{wN72}. This is hardly a surprise given that factorization theory has its origin in algebraic number theory, one of the pioneering works being~\cite{lC60}. More recently, the molecules of additive monoids such as numerical monoids and some generalizations of them have been studied in~\cite[Sections~3--4]{GG20}.
	
	\smallskip
	Let $\mathcal{A}(M)$ denote the set of atoms of a cancellative and commutative monoid $M$. Clearly, $\mathcal{A}(M)$ is contained in $\mathcal{M}(M)$. In this paper, we study the sizes of the sets of molecules that are not atoms in additive submonoids of $\qq_{\ge 0}$. One of the initial motivations of this project was the following realizability question posed by F. Gotti and the first author  in~\cite{GG20}.
	\begin{conjecture} \label{conj:NM molecule conjecture}
		For every $n \in \nn_{\ge 2}$ there exists a numerical monoid $N$ such that $|\mathcal{M}(N) \setminus \mathcal{A}(N)| = n$.
	\end{conjecture}
	\noindent With the statement of Conjecture~\ref{conj:NM molecule conjecture} in mind, we say that a class $\mathcal{C}$ of cancellative and commutative monoids is molecular if for every $n \in \nn_{\ge 2}$ there exists a monoid $M$ in $\mathcal{C}$ such that $|\mathcal{M}(M) \setminus \mathcal{A}(M)| = n$. Clearly, a molecular class must contain infinitely many non-isomorphic monoids. In the first part of this paper, we provide a positive answer to Conjecture~\ref{conj:NM molecule conjecture}, i.e., we prove that the class consisting of all numerical monoids is molecular.
	
	\smallskip
	Following D. D. Anderson, D. F. Anderson, and M. Zafrullah~\cite{AAZ90}, we say that an atomic monoid $M$ is a finite factorization monoid (or an FFM) if every element of $M$ has only finitely many factorizations, and we say that $M$ is a bounded factorization monoid (or a BFM) if for every element in $M$ there is a bound for the set of lengths of its factorizations. It is clear that
	\begin{equation} \label{diag:Anderson diagram fragment}
		UFM \Longrightarrow \ FFM \ \Longrightarrow \ BFM \ \Longrightarrow \ ATM,
	\end{equation}
	where ATM stands for atomic monoid. The chain of implications~\eqref{diag:Anderson diagram fragment} is a fragment of a larger diagram of atomic classes that first appeared in~\cite{AAZ90}, where it was illustrated that none of the implications in~\eqref{diag:Anderson diagram fragment} is reversible in the class consisting of integral domains. The original larger diagram of atomic classes of integral domains was further investigated in the sequel~\cite{AAZ92,AeA99,AA10}. It was recently proved in \cite[Theorem~4.3]{fG20a} that none of the implications in~\eqref{diag:Anderson diagram fragment} is reversible in the class consisting of semigroup rings $F[X;S]$, where $F$ is a field and $M$ is an additive monoid consisting of rationals.
	
	\smallskip
	A Puiseux monoid is an additive submonoid of $\qq_{\ge 0}$. It is well know that none of the implications in~\eqref{diag:Anderson diagram fragment} is reversible in the class of Puiseux monoids (see examples in Subsection~\ref{subsec:atomic classes of PMs}). Unlike numerical monoids, Puiseux monoids are not, in general, atomic. The second part of this paper is devoted to construct, for each implication in \eqref{diag:Anderson diagram fragment}, a molecular subclass of Puiseux monoids whose members witness the failure of the corresponding reverse implication. For instance, in Theorem~\ref{thm:the class C_2 is molecular} we construct a subclass of Puiseux monoids that is molecular, whose monoids are BFMs but not FFMs. We also construct a molecular class consisting of non-atomic Puiseux monoids.
	
	\medskip
	
	
	\section{Preliminary}
	\label{sec:Background and Notation}
	
	We let $\mathbb{N}$ and $\nn_0 := \nn \cup \{0\}$ denote the set of positive and nonnegative integers, respectively, and we let $\pp$ denote the set of primes. In addition, for $X \subseteq \rr$ and $r \in \rr$, we set $X_{\ge r} := \{x \in X \mid x \ge r\}$; in a similar way, we use the notations $X_{> r}$, $X_{\le r}$, and $X_{< r}$. Given $a, b \in \zz$, we refer by $\ldb a, b \rdb$ to the set $[a, b] \cap \zz$, i.e, the set of integers between $a$ and $b$. For each $q \in \qq_{> 0}$, the unique $n,d \in \nn$ such that $q = n/d$ and $\gcd(n,d) = 1$ are denoted by $\mathsf{n}(q)$ and $\mathsf{d}(q)$, respectively. We call $\mathsf{n}(q)$ and $\mathsf{d}(q)$ the \emph{numerator} and \emph{denominator} of $q$, respectively.
	\smallskip
	
	Throughout this paper, the term \emph{monoid} refers to a cancellative and commutative semigroup with identity. Since all monoids here are assumed to be commutative, we shall write them additively unless otherwise is specified. In addition, we shall tacitly assume that all monoids in this paper are reduced, in the sense that the only invertible element they contain is the identity element.
	\smallskip
	
	Let $M$ be a monoid. For a subset $S$ of $M$, we let $\langle S \rangle$ denote the smallest (under inclusion) submonoid of $M$ containing $S$. We say that $M$ is generated by $S$ if $M = \langle S \rangle$. In addition, $M$ is called \emph{finitely generated} if it can be generated by one of its finite subsets. An element $a \in M \setminus \{0\}$ is called an \emph{atom} if whenever $a = u + v$ for some $u,v \in M$ either $u = 0$ or $v = 0$. The set of atoms of $M$ is denoted by $\mathcal{A}(M)$, and $M$ is called \emph{atomic} if $M = \langle \mathcal{A}(M) \rangle$.
	\smallskip
	
	The free commutative monoid on $\mathcal{A}(M)$ is denoted by $\mathsf{Z}(M)$, and the elements of $\mathsf{Z}(M)$ are called \emph{factorizations} of $M$. If $z := a_1 + \dots + a_\ell \in \mathsf{Z}(M)$ for some $a_1, \dots, a_\ell \in \mathcal{A}(M)$, then $\ell$ is called the \emph{length} of the factorization $z$ and is denoted by~$|z|$. Since $\mathsf{Z}(M)$ is free, there exists a unique monoid homomorphism $\pi \colon \mathsf{Z}(M) \to M$ satisfying that $\pi(a) = a$ for all $a \in \mathcal{A}(M)$. For $x \in M$, the set $\mathsf{Z}(x) := \pi^{-1}(x)$ is called the \emph{set of factorizations} of $x$. Since $M$ need not be atomic, $\mathsf{Z}(x)$ may be empty for some $x \in M$.
	
	\begin{definition}
		Let $M$ be a monoid. An element $x \in M \setminus \{0\}$ is a \emph{molecule} if $|\mathsf{Z}(x)| = 1$. We let $\mathcal{M}(M)$ denote the set consisting of all molecules of $M$.
	\end{definition}
	
	For each $x \in M$, the set $\mathsf{L}(x) := \{|z| \mid z \in \mathsf{Z}(x)\}$ is called the \emph{set of lengths} of~$x$. Clearly, the set of lengths of a molecule is a singleton. Suppose now that $M$ is atomic. We say that $M$ is a \emph{UFM} (or a \emph{unique factorization monoid}) if every non-invertible element of $M$ is a molecule. In addition, $M$ is called an \emph{FFM} (or a \emph{finite factorization monoid}) if $\mathsf{Z}(x)$ is finite for all $x \in M$ while $M$ is called a \emph{BFM} (or a \emph{bounded factorization monoid}) if $\mathsf{L}(x)$ is finite for all $x \in M$. It is clear that every UFM is an FFM, every FFM is a BFM, and every BFM is atomic.
	\smallskip
	
	A submonoid $N$ of $(\nn_0,+)$ is said to be a \emph{numerical monoid}\footnote{Numerical monoids have been widely investigated under the term numerical semigroups.} if $\nn_0 \setminus N$ is a finite set. If $\nn_0 \setminus N$ is not empty, then $N$ is said to be a \emph{proper} numerical monoid; in this case, the maximum of $\nn_0 \setminus N$ is known as the \emph{Frobenius number} of $N$ and is denoted by $F(N)$. It is not hard to verify that a numerical monoid is always finitely generated and has a unique minimal set of generators, which is precisely its set of atoms. The \emph{embedding dimension} of $N$ is the cardinality of its generating set. As numerical monoids are finitely generated, they are FFMs \cite[Proposition~2.7.8]{GH06}, and so BFMs. Numerical monoids have been actively investigated (see~\cite{GR09} and references therein) and have many connections to several areas of mathematics (see~\cite{AG16} for some applications).
	\smallskip
	
	On the other hand, a submonoid $M$ of $(\qq_{\ge 0},+)$ is called a \emph{Puiseux monoid}. Unlike numerical monoids, Puiseux monoids may not be finitely generated or atomic: for instance, $M = \langle 1/2^n \mid n \in \nn_0 \rangle$ is clearly non-finitely generated and $\mathcal{A}(M)$ is empty. Further contrasts with numerical monoids are given by the existence of atomic Puiseux monoids that are not BFMs as it is the case of $\langle 1/p \mid p \in \pp \rangle$ (see Example~\ref{ex:atomic not BFM}) and the existence of Puiseux monoids that are BFMs but not FFMs as it is the case of $\nn_0 \cup \qq_{\ge n}$, where $n$ is a positive integer (see Example~\ref{ex:BFMs not FFMs}). Puiseux monoids have only been systematically studied recently in connection to factorization theory (see~\cite{CGG20,CGG20a} and references therein). In addition, Puiseux monoids have appeared in the literature in connection to commutative ring theory (see~\cite{CG19,aG74}) and, more recently, in the non-commutative context of monoids of matrices~\cite{BG20}.
	
	\medskip
	
	
	\section{Molecules of Interval Numerical Monoids}
	\label{sec:molecules in NM}
	
	In this first section we describe the set of molecules of numerical monoids generated by discrete intervals. To begin with, let us provide a formal definition.
	
	\begin{definition}
		We call a numerical monoid $N$ an \emph{interval numerical monoid} provided that $\mathcal{A}(N)$ consists of consecutive integers. For $a \in \nn$ and $n \in \ldb 0, a-1 \rdb$, we let $N_{a,n}$ denote the interval numerical monoid generated by the set $\{a+j \mid j \in \ldb 0,n \rdb \}$.
	\end{definition}
	
	Interval numerical monoids were first investigated by Garc\'ia-S\'anchez and Rosales in~\cite{GR99} where, among other results, they found a formula for the Frobenius number. They proved that for $a \in \nn$ and $n \in \ldb 0, a-1 \rdb$, the Frobenius number of the interval numerical monoid $N_{a,n}$ is $F(N_{a,n}) = \big\lceil \frac{a-1}{n} \big\rceil a - 1$.
	
	It follows immediately that $\mathcal{A}(N_{a,n}) = \{a+j \mid j \in \ldb 0,n \rdb\}$. However, the set of molecules of $N_{a,n}$ is not that easy to determine, and providing a full description of $\mathcal{M}(N_{a,n})$ is our primary purpose in this section.
	
	\begin{example}
		For every $a \in \nn_{\ge 2}$, the embedding-dimension-two numerical monoid $N_{a,1} = \langle a,a+1 \rangle$ is an interval numerical monoid. Clearly, $\mathcal{A}(N_{a,1}) = \{a,a+1\}$, and it is not hard to verify that
		\[
		\mathcal{M}(N_{a,1}) = \{ (m+n)a + n \mid m \in \ldb 0,a \rdb, \, n \in \ldb 0, a-1 \rdb, \, \text{and} \ (m, n) \neq  (0,0) \};
		\]
		for more information, see~\cite[Proof of Theorem~3.7]{GG20}.
	\end{example}
	
	In the next theorem we determine the set of molecules of $N_{a,n}$ for $a \in \nn_{\ge 3}$ and $n \in \ldb 1, a-1 \rdb$. First, we argue the following lemma.
	
	\begin{lemma} \label{lem:formulas for n_1 and n_2}
		Let $a, n \in \nn$ such that $a \ge 3$ and $n \in \ldb 1, a-1 \rdb$, and let $n_1$ and $n_2$ be the smallest positive integers such that $n_1 a \in \langle a + j \mid j \in \ldb 1,n \rdb \rangle$ and $n_2 (a+n) \in \langle a + j \mid j \in \ldb 0,n-1 \rdb \rangle$. Then
		\[
		n_1 = n_2 + 1 = \left\lceil \frac an \right\rceil + 1.
		\]
	\end{lemma}
	
	\begin{proof}
		Set $M = \langle  a+j\mid j\in  \ldb 0,n-1 \rdb \rangle$. If $x := n_{2}(a+n)\in M$, then $x$ has a factorization in $M$ with length greater than $n_2$. Now the inequality $(n_{2}+1)a \leq n_{2}(a+n)$ follows from the fact that $(n_2 + 1)a$ is the smallest element of $M$ that has a factorization of length greater than $n_{2}$. Therefore $a \le n_2 n$, which yields $n_2 \ge \lceil a/n \rceil$. It remains to prove that $\lceil a/n \rceil(a+n) \in M$. From $\ldb a, a+n-1 \rdb \subseteq M$, one can easily see that 
		\[
		\ldb (\lceil a/n \rceil + 1)a, (\lceil a/n \rceil + 1)(a + n - 1) \rdb \subseteq M.
		\] 
		Now since $a/n \le \lceil a/n \rceil$ and $\lceil a/n \rceil + 1 \le a + n$, we have that 
		\[
		(\lceil a/n \rceil + 1) a \le \lceil a/n \rceil (a+n) \le (\lceil a/n \rceil + 1)(a + n - 1). 
		\]
		Thus, $\lceil a/n \rceil(a+n) \in M$.
		We can prove that $n_1 = \lceil a/n \rceil + 1$ using similar arguments.
	\end{proof}
	
	We are now in a position to establish the main result of this section.
	
	\begin{theorem} \label{thm:molecules of generalized INM}
		Let $a, n \in \nn$ such that $a \ge 3$ and $n \in \ldb 1, a-1 \rdb$, and let $n_1$ and $n_2$ be the smallest positive integers such that $n_1 a \in \langle a + j \mid j \in \ldb 1,n \rdb \rangle$ and $n_2 (a+n) \in \langle a + j \mid j \in \ldb 0,n-1 \rdb \rangle$.  Then $\mathcal{M}(N_{a,n}) \setminus \mathcal{A}(N_{a,n}) = \mm_1 \cup \mm_2 \cup \mm'_2 \cup  \mm_3 \cup \mm_4 \cup \mm'_4$, where
		\begin{itemize}
			\item $\mm_1 = \{ja \mid j \in \ldb 2, n_1 - 1 \rdb \}$,
			\smallskip
			
			\item $\mm_2 = \{ja + (a+1) \mid j \in \ldb 1, n_1 - 2 \rdb\}$,
			\smallskip
			
			\item $\mm'_2 = \{(n_1 - 1)a + (a+1)\}$ if $a/n \in \zz$, and $\mm'_2 = \emptyset$ if $a/n \notin \zz$,
			\smallskip
			
			\item $\mm_3 = \{j(a+n) \mid j \in \ldb 2, n_2 - 1 \rdb \}$,
			\smallskip
			
			\item $\mm_4 = \{j(a+n) +(a+n-1) \mid j \in \ldb 1, n_2 - 2 \rdb \}$, and
			\smallskip
			
			\item $\mm'_4 = \{ (n_2 - 1)(a+n) + (a+n-1) \}$ if $a/n \in \zz$, and $\mm'_4 = \emptyset$ if $a/n \notin \zz$.
		\end{itemize}
	\end{theorem}
	
	\begin{proof}
		The element $2a$ is clearly a molecule. In addition, if $ja$ is a molecule for some $j < n_1 - 1$, then $(j+1)a \notin \langle a + j \mid j \in \ldb 1,n \rdb \rangle$, and so any factorization of $(j+1)a$ yields a factorization of $ja$ (after canceling one copy of $a$), whence $(j+1)a$ is also a molecule. We have proved inductively that $\mm_1$ consists of molecules of $N_{a,n}$.
		\smallskip
		
		To verify that each element of $\mm_2$ is also a molecule, fix $k \in \ldb 1, n_1 - 2 \rdb$. Let $z := \sum_{i=0}^n c_i (a+i)$ be a factorization of $ka + (a+1)$, where $c_0, \dots, c_n \in \nn_0$. Then $c_j > 0$ for some $j \in \ldb 1, n \rdb$, and so
		\[
		z' := (c_{j-1} + 1)(a + (j-1)) + (c_j - 1)(a + j) + \! \! \sum_{i \in \ldb 0,n \rdb \setminus \{j-1,j\}} \! \!c_i(a + i)
		\]
		is a factorization of $(k+1)a$. Since $(k+1)a$ is a molecule (as proved in the previous paragraph), $z' = (k+1)a$ in $\mathsf{Z}(N_{a,n})$, and so the equalities $a + (j-1) = a$ and $c_{j-1} + 1 = k+1$ hold. Thus, $j=1$ and $c_0 = c_{j-1} = k$. As $c_1 \ge 1$, the equality $c_0 = k$ forces the equalities $c_1 = 1$ and $c_i = 0$ for every $i \ge 2$. Then $z = ka + (a+1)$ and, therefore, $ka + (a+1)$ must be a molecule. So $\mm_2$ also consists of molecules.
		\smallskip
		
		Let us check that the singleton $\mm'_2$ contains a molecule when $n$ divides $a$. Write $a = kn$ for some $k \in \nn_{\ge 2}$. It follows from Lemma~\ref{lem:formulas for n_1 and n_2} that $k = a/n = n_1 - 1$. Notice, on the other hand, that if the element $(n_1 - 1)a + (a+1)$ has a factorization $z$ different from the obvious one, such a factorization must have length at most $k$. As a result, $k(a+n) = (k+1)a < \pi(z) \le k(a+n)$, which is not possible. Thus, $(n_1 - 1)a + (a+1)$ must be a molecule.
		\smallskip
		
		Verifying that $\mm_3$ (resp., $\mm_4$ and $\mm'_4$) consists of molecules can be done following the same lines we just used to argue that $\mm_1$ (resp., $\mm_2$ and $\mm'_2$) consists of molecules. Therefore $\mm_1 \cup \mm_2 \cup \mm'_2 \cup \mm_3 \cup \mm_4 \cup \mm'_4$ is a subset of $\mathcal{M}(N_{a,n}) \setminus \mathcal{A}(N_{a,n})$.
		\smallskip
		
		To prove that the reverse inclusion holds, take $m \in \mathcal{M}(N_{a,n}) \setminus \mathcal{A}(N_{a,n})$ and let $z := \sum_{i=0}^n c_i (a+i)$ be the only factorization of $m$, where $c_0, \dots, c_n \in \nn_0$. Since $m$ is not an atom, $\sum_{i=0}^n c_i \ge 2$.
		
		Suppose for the sake of a contradiction that $c_j \ge 1$ for some $j \notin \{0,1,n-1,n\}$. If $c_j \ge 2$, then one could replace $2(a+j)$ in $z$ by $(a+(j-1)) + (a + (j+1))$ to obtain a factorization of $m$ different from $z$. Hence $c_j = 1$, and so there exists $c_k > 0$ for some $k \neq j$. We first assume that $k < j$. If $k = 0$, then as $a + 1 \neq a + j$ one could replace $a + (a + j)$ by $(a+1) + (a + (j-1))$ to obtain a factorization of $m$ different from $z$. On the other hand, if $k > 0$, then after replacing $(a+k) + (a+j)$ by $(a+(k-1)) + (a + (j+1))$ we would obtain again a factorization of $m$ different from $z$. As a result, $k < j$ generates contradictions. The case of $k > j$ can be handled \emph{mutatis mutandis} to generate contradictions. Thus, $c_j = 0$ when $j \notin \{0,1,n-1,n\}$. Let us split the rest of the proof into the following two cases.
		\smallskip
		
		CASE 1: $n \ge 3$. Write $z = c_0 a + c_1 (a+1) + c_{n-1}(a + (n-1)) + c_n(a+n)$. Reasoning as in the previous paragraph, one finds that either $c_0 = c_1 = 0$ or $c_{n-1} = c_n = 0$. Suppose first that $c_{n-1} = c_n = 0$, and so $z = c_0 a + c_1 (a+1)$. In this case, $c_0 \le n_1 - 1$ and $c_1 \in \{0,1\}$, as otherwise one could replace $2(a+1)$ by $a + (a+2)$ to obtain a factorization of $m$ different from $z$. If $c_1 = 0$, it is clear that $c_0a \notin \langle a + j \mid j \in \ldb 1,n \rdb \rangle$, and so $m \in \mm_1$. Then suppose that $c_1 = 1$. If $c_0 < n_1 - 1$, then $m \in \mm_2$. Accordingly, suppose that $c_0 = n_1 - 1$ and so that $m = n_1 a + 1$.
		
		Because $n_1 a \in \langle a+j \mid j \in \ldb 1,n \rdb \rangle$ we can replace $n_1 a$ by a sum of atoms in $\{a+1, \dots, a + n\}$ in $m = n_1 a + 1$. Provided that $a+n$ does not divide $n_1 a$, this will yield a factorization of $m$ different from $z$. Therefore $n_1 a = n'(a+n)$ for some $n' \in \nn$. By Lemma~\ref{lem:formulas for n_1 and n_2} and the minimality of $n_2$, one sees that $n' \ge n_2 = n_1 - 1$. This, along with $n_1 a = n'(a+n)$, guarantees that $n' = n_1 - 1$. Then $n_1 a = (n_1 - 1)(a+n)$ can be rewritten as $a = (n_1 - 1)n$. Hence $n$ divides $a$, and so $m \in \mm'_2$.
		
		The case of $c_0 = c_1 = 0$ follows analogously.
		\smallskip
		
		CASE 2: $n=2$. Write now $z = c_0a + c_1(a+1) + c_2(a+2)$. As $m$ is a molecule, $c_0 c_2 = 0$ and $c_1 \le 1$. Assume first that $c_2 = 0$. Then $z = c_0a + c_1(a+1)$, and one can check that $m \in \mm_1 \cup \mm_2 \cup \mm'_2$ as we did in CASE 1. If $c_2 \neq 0$, then $c_0 = 0$ and so $z = c_1(a+1) + c_2(a+2)$. This case can be handled similarly to the case of $z = c_0a + c_1(a+1)$ to conclude that $m \in \mm_3 \cup \mm_4 \cup \mm'_4$.
	\end{proof}
	
	\begin{cor} \label{cor:molecules of interval NMs}
		Take $a, n \in \nn$ such that $a \ge 3$ and $n \in \ldb 1, a-1 \rdb$. Then
		\[
		|\mm(N_{a,n})| = \twopf{4 \lceil a/n \rceil + n - 5}{n \nmid a}{4 \lceil a/n \rceil + n - 3}{n \mid a}.
		\]
	\end{cor}
	
	\medskip
	
	
	\section{Molecularity of the Class of Numerical Monoids}
	\label{sec:Molecularity of NMs}
	
	In this section, we formally introduce the fundamental classes of Puiseux monoids we shall be concerned with, and then we prove Conjecture~\ref{conj:NM molecule conjecture}.
	
	\smallskip
	\subsection{Atomic Classes of Puiseux Monoids}
	\label{subsec:atomic classes of PMs}
	
	There are three classes of atomic Puiseux monoids that we will present in this subsection with the intention of later investigating the sets of molecules of their members.
	\medskip
	
	Let $\mathcal{C}_1$ denote the class of all Puiseux monoids that are FFMs but not UFMs.
	
	\begin{example}
		Let $M$ be a finitely generated Puiseux monoid. It follows from \cite[Proposition~2.7.8]{GH06} that $M$ is an FFM, and it follows from \cite[Proposition~4.3.1]{fG20} that $M$ is a UFM if and only if $M \cong (\nn_0,+)$. In particular, each nontrivial numerical monoid belongs to $\mathcal{C}_1$. 
	\end{example}
	
	The class $\mathcal{C}_1$ also contains non-finitely generated monoids.
	
	\begin{example}
		Consider the Puiseux monoid $M = \langle (3/2)^n \mid n \in \nn_0 \rangle$. Since $M$ can be generated by an increasing sequence, namely the increasing powers of $3/2$, it follows from \cite[Theorem~5.6]{fG19} that $M$ is an FFM. In addition, \cite[Proposition~4.3.1]{fG20} guarantees that $M$ is not a UFM. Since $\mathcal{A}(M) = \{ (3/2)^n \mid n \in \nn_0 \}$ by \cite[Proposition~4.3]{CGG20a}, the monoid $M$ is a non-finitely generated monoid that belongs to $\mathcal{C}_1$.
	\end{example}
	
	Let $\mathcal{C}_2$ denote the class of all Puiseux monoids that are BFMs but not FFMs. It is clear that the classes $\mathcal{C}_1$ and $\mathcal{C}_2$ are disjoint.
	
	\begin{example} \label{ex:BFMs not FFMs}
		For every $n \in \nn$ it is clear that $M = \nn_0 \cup \qq_{\ge n}$ is a Puiseux monoid. As $0$ is not a limit point of $M^\bullet$, it follows from \cite[Proposition~4.5]{fG19} that $M$ is a BFM. In addition, one can easily check that $\mathcal{A}(M) = \{1\} \cup \big( \qq \cap (n,n+1) \big)$. Now notice that for every $m \in \nn_{\ge 2}$ the equalities 
		\[
		2n+1 = (n + 1/m) + (n + 1 - 1/m)
		\]
		determine infinitely many distinct factorizations in $\mathsf{Z}_{M}(2n+1)$. Hence $M$ is not an FFM. Therefore the class $\mathcal{C}_2$ contains infinitely many Puiseux monoids.
	\end{example}
	
	Finally, we denote by $\mathcal{C}_3$ the class of all atomic Puiseux monoids that are not BFMs. Clearly, the class $\mathcal{C}_3$ is disjoint from $\mathcal{C}_1 \cup \mathcal{C}_2$.
	
	\begin{example} \label{ex:atomic not BFM}
		Let $P$ be a set containing infinitely many prime numbers, and consider the Puiseux monoid $M = \langle 1/p \mid p \in P \rangle$. It is not hard to verify that $M$ is atomic with $\mathcal{A}(M) = \{1/p \mid p \in P\}$ (see \cite[Theorem 4.5]{CGG20a}). On the other hand, $M$ is not a BFM because the fact that $p \frac{1}{p} \in \mathsf{Z}(1)$ for every $p \in P$ implies that $P \subseteq \mathsf{L}(1)$. So the class $\mathcal{C}_3$ contains infinitely many members.
	\end{example}
	\smallskip
	
	The fundamental questions we are interested in are related to the size of the set $\mathcal{M}(M) \setminus \mathcal{A}(M)$, where $M$ is a Puiseux monoid. In particular, we would like to know what are the possible sizes of the set $\mathcal{M}(M) \setminus \mathcal{A}(M)$ in the distinct classes of Puiseux monoids determined by the chain of implications~\ref{diag:Anderson diagram fragment}.
	
	\smallskip
	
	
	\subsection{A Conjecture on Molecularity}
	
	We say that a class $\mathcal{C}$ of monoids is \emph{molecular} if for every $n \in \nn_{\ge 2}$, there exists a monoid $M$ in $\mathcal{C}$ such that $|\mathcal{M}(M) \setminus \mathcal{A}(M)| = n$. Clearly, Conjecture~\ref{conj:NM molecule conjecture} can be rephrased by saying that the class consisting of all numerical monoids is molecular. We now offer a proof of this conjecture.
	
	\begin{theorem} \label{thm:molecules realization for NM}
		The class of numerical monoids is molecular.
	\end{theorem}
	
	\begin{proof}
		Let $\mathcal{N}$ denote the class consisting of all proper numerical monoids, and for $N$ in $\mathcal{N}$ set $m(N) := |\mathcal{M}(N) \setminus \mathcal{A}(N)|$. Let us prove that $S := \{m(N) \mid N \in \mathcal{N} \} = \nn_{\ge 2}$. It is easy to see that $S \subseteq \nn_{\ge 2}$. For the reverse inclusion, fix $s \in \nn_{\ge 2}$. Observe that if we set $n = 2$ in Corollary~\ref{cor:molecules of interval NMs}, then we obtain that $|\mm(N_{a,2}) \setminus \mathcal{A}(N_{a,2})| = 2a - 4$. Therefore if $s$ is even with $s \ge 6$ we have that $|\mm(N_{s/2,2}) \setminus \mathcal{A}(N_{s/2,2})| = s-4$. As a result, $S$ contains all even numbers in $\nn_{\ge 2}$.
		
		To prove that $S$ contains all odd numbers in $\nn_{\ge 2}$, we will describe the sets of molecules of the numerical monoids $N_a := \langle a, a+1, a+3, a+4 \rangle$ for every $a \in \nn_{\ge 6}$. Fix $a \in \nn_{\ge 6}$. It is clear that $\mathcal{A}(N_a) = \{a, a+1, a+3, a+4\}$. For each $n \in \nn$, we set
		\[
		I_n := \{x \in N_a \mid n \in \mathsf{L}(x)\}.
		\]
		Notice that $I_1 = \mathcal{A}(N_a)$, and $I_j = \ldb ja, j(a+4) \rdb$ for every $j \in \nn_{\ge 2}$. 
		
		Out of the nine elements in $I_2$, only $2a+4$ has different factorizations of length~$2$.
		
		We shall prove inductively that for every $n \in \nn_{\ge 3}$ each element in the discrete interval $\ldb na+3, n(a+4) - 3 \rdb \subseteq I_n$ has at least two factorizations of length $n$ (in particular, they will fail to be molecules). The elements $3a+4, 3a+5, 3a+7, 3a+8 \in I_3$ can be written as the addition of $2a+4$ and the atoms $a,a+1,a+3,a+4$, respectively. Therefore each of them has at least two different factorizations of length $3$. In addition, one can readily verify that each of the elements $3a+3, 3a+6, 3a+9 \in I_3$ has at least two different factorizations of length~$3$. Suppose that for $n \ge 4$, each of the elements $(n-1)a+3, (n-1)a+4, \dots, (n-1)(a+4)-3 \in I_{n-1}$ has at least two different factorizations of length $n-1$.
		Adding $a$ to each of these elements, we find that each of the elements $na+3, na+4, \dots, n(a+4) - 7 \in I_n$ has at least two different factorizations of length $n$. Furthermore, $n(a+4) - 6, n(a+4) - 5 \in I_n$ can be written as the addition of the element $(n-1)(a+4) - 5$ and the atoms $a+3$ and $a+4$, respectively. Similarly, $n(a+4) - 4, n(a+4) - 3$ can be written as the addition of the element $(n-1)(a+4) - 3$ and the atoms $a+3$ and $a+4$, respectively. Hence each of the elements $n(a+4) - 6, n(a+4) - 5, n(a+4) - 4, n(a+4) - 3 \in I_n$ also has at least two different factorizations of length $n$. This concludes our inductive argument. Lastly, it is clear that for every $n \in \nn_{\ge 3}$ the smallest three and the largest three elements of the discrete interval $I_n$ each has exactly one factorization of length $n$.
		
		Now we just need to determine the elements of $N_a$ that belong to more than one~$I_j$. Notice that if $x \in I_j \cap I_k$ for some $j < k$, then $x \in I_j \cap I_{j+1}$. When this is the case, one finds that $j(a+4) = \max I_j \ge \min I_{j+1} = (j+1)a$, and so $4j \ge a$. Set
		\[
		m := \min\{j \in \nn \mid I_j \cap I_{j+1} \neq \emptyset\}.
		\]
		Since $\max I_j - \min I_j = 4j$ and $\min I_{j+1} - \min I_j = a$ for every $j \in \nn$, the three largest elements of $I_j$ will be contained in $I_{j+1}$ for every $j \ge m+1$. Similarly, the smallest three elements of $I_j$ will be contained in $I_{j-1}$ for every $j \ge m+2$.
		
		Next we find the size $m(N_a)$ of $\mm(N_a) \setminus \mathcal{A}(N_a)$ when $a \not\equiv 1 \pmod 4$. First, we assume that $m > 2$. In this case, there are $8$ molecules in $I_2$, there are $6(m-3)$ molecules in $I_3 \cup \dots \cup I_{m-1}$, there are $3 + \big(3 - (4m-a + 1)\big)$ molecules in $I_m$, and there are $\big(3 - (4m-a + 1)\big)$ molecules in $I_{m+1}$. Hence
		\[
		m(N_a) = 8 + 6(m-3) + 3 + 2(2 - (4m - a)) = 2a - 2m - 3 = 2a - 2 \big\lceil \frac{a}{4} \big\rceil -3.
		\]
		In the case of $m=2$, we have that $a \in \{6,7,8\}$ and it can be readily seen for each of such values of $a$, that the equality $m(N_a) = 2a - 2 \big\lceil \frac{a}{4} \big\rceil -3$ holds.
		
		Finally, we are ready to verify that every odd number in $\nn_{\ge 3}$ belongs to $S$. To do so, take $a = 4k - i$ for $0 \le i \le 2$ (and therefore, $k \ge 2$), and observe that 
		\begin{align*}
			1 + 2\nn_{\ge 2} &=  \{(6k+i)-8 \mid k \in \nn_{\ge 2}, \ i \in \{1,3,5\} \} \\
			&= \{ m(N_{4k-i}) \mid k \in \nn_{\ge 2}, \ i \in \{0,1,2\} \} \\
			&\subseteq \{m(N_a) \mid a \in \nn_{\ge 6}\} \subseteq S.
		\end{align*}
		To verify that $3$ also belongs to $S$, it suffices to observe that $\mm(\langle 2,3 \rangle) = \{4,5,7\}$. We conclude that $\nn_{\ge 2} \subseteq S$, and so $\{|\mathcal{M}(N) \setminus \mathcal{A}(N)| : N \in \mathcal{N} \} = \nn_{\ge 2}$.
	\end{proof}
	
	As $\mathcal{C}_1$ contains every numerical monoid, one obtains the following corollary.
	
	\begin{cor} \label{cor:the class C_1 is molecular}
		The class $\mathcal{C}_1$ is molecular.
	\end{cor}
	
	\medskip
	
	
	\section{Molecularity of Further Classes of Puiseux Monoids}
	
	In this section, we turn our attention to classes of non-finitely generated Puiseux monoids, and study them in the same direction we studied the class of numerical monoids in Section~\ref{sec:Molecularity of NMs}.
	
	\smallskip
	
	
	\subsection{Molecularity of $\mathcal{C}_2$}
	
	We have seen in Corollary~\ref{cor:the class C_1 is molecular} that the class $\mathcal{C}_1$ is molecular. We proceed to provide a similar result for the class $\mathcal{C}_2$.
	
	\begin{theorem} \label{thm:the class C_2 is molecular}
		The class $\mathcal{C}_2$ is molecular.
	\end{theorem}
	
	\begin{proof}
		Fix $n \in \nn_{\ge 2}$. By Theorem~\ref{thm:molecules realization for NM}, there exists a numerical monoid $N$ satisfying $|\mathcal{M}(N) \setminus \mathcal{A}(N)| = n$. Set $a_0 := \min \mathcal{M}(N)$ and $K = \max \mathcal{M}(N)$. Therefore $\delta := a_0/K$ satisfies $0 < \delta < 1$. Now consider the Puiseux monoid $M = \langle A \rangle$, where $A = \bigcup_{a \in \mathcal{A}(N)} [a, a + \delta) \cap \qq$. As $0$ is not a limit point of $M^\bullet$, it follows from \cite[Proposition~4.5]{fG19} that $M$ is a BFM and, in particular, an atomic monoid. We proceed to verify that $\mathcal{A}(M) = A$.
		
		To do this, take $x \in M \cap \nn$ with $x \le K$, and write $x = a_1 + \dots + a_\ell$ for $\ell \in \nn$ and $a_1, \dots, a_\ell \in A$. Then suppose, by way of contradiction, that $m := \max \{a_i - \lfloor a_i \rfloor \mid i \in \ldb 1,\ell \rdb\} > 0$. In this case, the fact that $x \in \nn$ guarantees the last inequality of
		\begin{equation} \label{eq:auxiliary 1}
			\ell \delta > \ell m \ge \sum_{i=1}^\ell (a_i - \lfloor a_i \rfloor) \ge 1.
		\end{equation}
		Using~\eqref{eq:auxiliary 1} we obtain that $x \ge \ell a_0 = \ell \delta K > K$, which is a contradiction. As a consequence, $a_1, \dots, a_\ell \in N$.
		
		Now take $a \in A$ and write $a = a_1 + \dots + a_\ell$ for $\ell \in \nn$ and $a_1, \dots, a_\ell \in A$. Because $\sum_{i=1}^\ell (a_i - \lfloor a_i \rfloor ) \ge  a - \lfloor a \rfloor$, we can write $\lfloor a \rfloor = b_1 + \dots + b_\ell$, where $b_i \in [ \lfloor a_i \rfloor, a_i] \cap \qq$ for every $i \in \ldb 1, \ell \rdb$. The inclusion $\lfloor a \rfloor \in M \cap \nn_{\le K}$, along with our argument in the previous paragraph, now implies that $b_1, \dots, b_\ell \in N$. As $\lfloor a \rfloor \in \mathcal{A}(N)$, it follows that $\ell = 1$, and so $a \in \mathcal{A}(M)$. Hence $\mathcal{A}(M) = A$.
		
		Since $M$ is a BFM, proving that $M$ belongs to $\mathcal{C}_2$ amounts to verifying that $M$ is not an FFM. To do so, take $r \in (0, \delta) \cap \qq$ and set $x := 2 a_0 + r$. It is clear that $x \in M$. On the other hand, it can be readily seen that $a_0 + \big( r/2 \pm 1/n \big) \in \mathcal{A}(M)$ for every $n \in \nn$ with $n \ge 2/r$. As a result, the equalities $x = \big( a_0 + \big(\frac{r}{2} - \frac{1}{n} \big) \big) + \big( a_0 + \big(\frac{r}{2} + \frac{1}{n} \big) \big)$ (for every $n \in \nn_{\ge 2/r}$) yield infinitely many factorizations of $x$. As a consequence, $M$ is not an FFM.
		
		Next we show that no element $x \in M \setminus (\mathcal{A}(M) \cup N)$ belongs to $\mathcal{M}(M)$. Since $x \notin \mathcal{A}(M) \cup N$, there exists $z \in \mathsf{Z}(x)$ having a length-$2$ subfactorization $a_1 + a_2$ such that $r := a_1 - \lfloor a_1 \rfloor > 0$. Take $n \in \nn$ such that $1/n < \min \{r, \lfloor a_2 \rfloor - a_2 + \delta\}$, and note that $a_1 - 1/n, a_2 + 1/n \in \mathcal{A}(M)$. Then, after replacing $a_1 + a_2$ in $z$ by $\big( a_1 - \frac{1}{n} \big) + \big( a_2 + \frac{1}{n} \big)$, one obtains a factorization of $x$ different from $z$. Hence $x$ is not a molecule.
		
		The inclusion $\mathcal{A}(N) \subseteq \mathcal{A}(M)$, together with our argument in the previous paragraph, ensures that $\mathcal{M}(M) \setminus \mathcal{A}(M) \subseteq \mathcal{M}(N)$. On the other hand, if $x \in \mathcal{M}(N)$, then $x \in M \cap \nn_{\le K}$ and so each factorization of $x$ in $M$ is also a factorization of $x$ in $N$. Hence $x$ must belong to $\mathcal{M}(M)$. As a result, $\mathcal{M}(M) \setminus \mathcal{A}(M) = \mathcal{M}(N) \setminus \mathcal{A}(N)$, from which we conclude that $|\mathcal{M}(M) \setminus \mathcal{A}(M)| = n$.
	\end{proof}
	
	Unlike the case of numerical monoids, there are monoids $M$ in $\mathcal{C}_2$ satisfying $|\mathcal{M}(M) \setminus \mathcal{A}(M)| = 1$. For instance, consider for every $n \in \nn$ the monoid $\{0\} \cup \qq_{\ge n}$. It is easy to see that the only molecule of $\{0\} \cup \qq_{\ge n}$ that is not an atom is $2n$.
	\smallskip
	
	In the direction of Corollary~\ref{cor:the class C_1 is molecular} and Theorem~\ref{thm:the class C_2 is molecular}, we have the following conjecture.
	
	\begin{conjecture}
		The class $\mathcal{C}_3$ is molecular.
	\end{conjecture}
	
	\smallskip
	
	
	\subsection{Molecularity of Non-Atomic Puiseux Monoids}
	
	Each of the monoids we have treated so far is atomic. However, there are plenty of non-atomic Puiseux monoids. Let~$\mathcal{C}_4$ consist of all non-atomic Puiseux monoids. In the same direction of our previous results, we have the following proposition.
	
	\begin{proposition} \label{prop:the class of non-atomic PMs is molecular}
		The class $\mathcal{C}_4$ is molecular.
	\end{proposition}
	
	\begin{proof}
		Fix $n \in \nn_{\ge 2}$. By Theorem~\ref{thm:molecules realization for NM} there exists a numerical monoid $N$ such that $|\mathcal{M}(N) \setminus \mathcal{A}(N)| = n$. Take $p \in \pp$ such that $p > \max \mathcal{M}(N)$, and consider the Puiseux monoid
		\[
		M = \bigg\langle N \cup \bigg\{ \frac{p}{2^n} \ \bigg{|} \ n \in \nn \bigg\} \bigg\rangle.
		\]
		Clearly, none of the elements of $M$ of the form $p/2^n$ belongs to $\mathcal{A}(M)$. Therefore $\mathcal{A}(M) \subseteq \mathcal{A}(N)$. Furthermore, since $p > \max \mathcal{M}(N) \ge \max \mathcal{A}(N)$, one can readily verify that $\mathcal{A}(M) = \mathcal{A}(N)$. Because $\mathcal{A}(M)$ is a finite set and $0$ is a limit point of~$M^\bullet$, the Puiseux monoid $M$ cannot be atomic, that is, $M$ belongs to $\mathcal{C}_4$. Moreover, $\mathcal{A}(M) = \mathcal{A}(N)$ implies that $\mathsf{Z}_N(x) = \mathsf{Z}_M(x)$ for all $x \in N$. Hence a molecule of $N$ remains a molecule in $M$, i.e., $\mathcal{M}(N) \subseteq \mathcal{M}(M)$. On the other hand, if $x \in \mathcal{M}(M)$, then $\mathsf{Z}_M(x)$ is nonempty and, therefore, $\mathcal{A}(M) = \mathcal{A}(N)$ ensures that $x \in N$. As  $\mathsf{Z}_N(x) = \mathsf{Z}_M(x)$, it follows that $x \in \mathcal{M}(N)$. Thus, $\mathcal{M}(M) \subseteq \mathcal{M}(N)$, and so we obtain that $\mathcal{M}(M) = \mathcal{M}(N)$. This, together with the fact that $\mathcal{A}(M) = \mathcal{A}(N)$, ensures that $|\mathcal{M}(M) \setminus \mathcal{A}(M)| = n$, which concludes our proof because $n$ was taken arbitrarily in the set $\nn_{\ge 2}$.
	\end{proof}
	
	\medskip
	
	
	\section{Infinite Molecularity}
	
	This last section is devoted to explore the extreme case when the set $\mathcal{M}(M) \setminus \mathcal{A}(M)$ has infinite cardinality (as before, $M$ is taken to be a Puiseux monoid). This motivates the question as to whether one can find Puiseux monoids satisfying this property in each of the $\mathcal{C}_i$ classes introduced in previous sections. In order to address this question, the following definition is pertinent.
	
	\begin{definition}
		We say that a Puiseux monoid $M$ is \emph{infinitely molecular} provided that $|\mathcal{M}(M) \setminus \mathcal{A}(M)| = \infty$.
	\end{definition}
	
	As we shall reveal in the next propositions, each of the $\mathcal{C}_i$ classes contains an infinite subclass consisting of non-isomorphic Puiseux monoids that are infinitely molecular.
	
	\begin{proposition} \label{prop:C_1 is infinite molecular}
		There exists an infinite subclass of $\mathcal{C}_1$ consisting of infinitely molecular Puiseux monoids.
	\end{proposition}
	
	\begin{proof}
		Consider for every $r \in \qq_{> 1} \setminus \nn$ the Puiseux monoid $M_r := \langle r^n \mid n \in \nn_0 \rangle$. It follows from~\cite[Proposition~4.3]{CGG20a} that $M_r$ is atomic with $\mathcal{A}(M_r) = \{r^n \mid n \in \nn_0\}$. Indeed, since $M_r$ is generated by the increasing sequence $(r^n)_{n \in \nn_0}$, it follows from \cite[Theorem~5.6]{fG19} that $M_r$ is an FFM. Notice that $M_r$ is not a UFM; for instance, $\mathsf{n}(r)1$ and $\mathsf{d}(r)r$ are two distinct factorizations in $\mathsf{Z}(\mathsf{n}(r))$. Then $M_r$ belongs to $\mathcal{C}_1$.
		
		Now it follows from \cite[Lemma~3.2]{CGG20} that, for every $n \in \nn$, the element $1 + r^n$ is a molecule of $M_r$ and, therefore, $1 + r^n \in \mathcal{M}(M_r) \setminus \mathcal{A}(M_r)$. Hence $|\mathcal{M}(M_r) \setminus \mathcal{A}(M_r)| = \infty$. Finally, take $t \in \qq_{>1} \setminus \nn$ such that $M_r$ and $M_t$ are isomorphic monoids. It follows from~\cite[Proposition~3.2]{fG18} that $M_t = q M_r$ for some $q \in \qq_{>0}$. Since multiplication by $q$ is an increasing function, it must take $1 = \min M_r^\bullet$ to $1 = \min M_t^\bullet$. Then $q = 1$, and so $t=r$. As a result, $\mathcal{C}_1$ contains infinitely many non-isomorphic infinitely molecular Puiseux monoids.
	\end{proof}
	
	We proceed to show that the class $\mathcal{C}_2$ also contains plenty of infinitely molecular Puiseux monoids.
	
	\begin{proposition} \label{prop:C_2 is infinite molecular}
		There exists an infinite subclass of $\mathcal{C}_2$ consisting of infinitely molecular Puiseux monoids.
	\end{proposition}
	
	\begin{proof}
		Let $P$ be an infinite set of primes, and let $(p_n)_{n \in \nn}$ be a strictly increasing sequence with underlying set $P$. For each $n \in \nn$, let $R_n$ denote the localization of the ring $\zz$ at the multiplicative monoid generated by $\{p_1, \dots, p_n\}$ and let $M_n$ be the Puiseux monoid $\{0\} \cup (R_n \cap \qq_{\ge n + 1/p_n})$. Now set $M_P = \bigcup_{n \in \nn_0} M_n$, where $M_0$ is taken to be~$\nn_0$. We can readily verify that $M_P$ is closed under addition, whence it is a Puiseux monoid. As $\min M_P^\bullet = 1$, it follows that $0$ is not a limit point of $M_P^\bullet$, and so $M_P$ is a BFM by \cite[Proposition~4.5]{fG19}. To argue that $M_P$ is not an FFM, first notice that every element in $R_n \cap [n+ 1/p_n, n+1 + 1/p_n)$ whose denominator is divisible by $p_n$ is an atom of $M_P$. Since
		\[
		\bigg( n + \frac{1}{p_n} + \frac{1}{p_n^k} \bigg) + \bigg( n + 1 + \frac{1}{p_n} - \frac{1}{p_n^k} \bigg) \in \mathsf{Z}_{M_P}\bigg(2n + 1 + \frac{2}{p_n} \bigg),
		\]
		for every $k \in \nn_{\ge 2}$, it follows that $2n + 1 + \frac{2}{p_n}$ is an element of $M_P$ with infinitely many factorizations. Therefore $M_P$ is not an FFM and, as a result, $M_P$ belongs to the class~$\mathcal{C}_2$.
		
		To prove that $M_P$ is infinitely molecular, consider the set $S := \{\frac{np_n + p_n + 1}{p_n} \mid n \in \nn\}$. Since $\frac{np_n + p_n + 1}{p_n} = 1 + \big( n + \frac{1}{p_n}\big)$, we find that $S$ is a subset of $M_P$ consisting of elements that are not atoms. Observe that none of the elements of $M_P$ that is strictly less than $n + \frac{1}{p_n}$ has a denominator divisible by $p_n$. This, along with the fact that whenever $a_1 + \dots + a_k$ is a factorization of $\frac{np_n + p_n + 1}{p_n}$ the prime $p_n$ must divide $\mathsf{d}(a_i)$ for some $i \in \ldb 1,k \rdb$, implies that the only factorization of $\frac{np_n + p_n + 1}{p_n}$ must be $1 + \big( n + \frac{1}{p_n}\big)$. As a result, $S$ is an infinite set of molecules that are not atoms, which implies that $M_P$ is an infinitely molecular Puiseux monoid. Finally, write $\pp$ as the disjoint union of countably many disjoint infinite sets, namely, $\pp = \bigcup_{n \in \nn} P_n$. It follows immediately from~\cite[Proposition~3.2]{fG18} that the monoid $M_{P_n}$ and $M_{P_m}$ are not isomorphic when $m \neq n$. Hence $\mathcal{C}_2$ contains an infinite subclass of non-isomorphic infinitely molecular Puiseux monoids.
	\end{proof}
	
	The class $\mathcal{C}_3$ also contains infinitely many non-isomorphic Puiseux monoids that are infinitely molecular. To argue this, we use a subfamily of the monoids introduced in Example~\ref{ex:atomic not BFM}.
	
	\begin{proposition}
		There exists an infinite subclass of $\mathcal{C}_3$ consisting of infinitely molecular Puiseux monoids.
	\end{proposition}
	
	\begin{proof}
		Let $P$ be an infinite set of odd primes, and consider the Puiseux monoid $M_P = \langle 1/p \mid p \in P \rangle$ introduced in Example~\ref{ex:atomic not BFM}. We have already seen that $M_P$ belongs to the class $\mathcal{C}_3$. In addition, it follows from \cite[Proposition~4.10]{GG20} that $2/p$ is a molecule of $M_P$ for every $p \in P$ and, therefore, $\{2/p \mid p \in P \}$ is an infinite set of molecules of $M_P$ that are not atoms. Hence $M_P$ is infinitely molecular. Mimicking our argument in the proof of Proposition~\ref{prop:C_2 is infinite molecular}, we can argue that the construction used in this proof yields infinitely many non-isomorphic infinitely molecular Puiseux monoids in $\mathcal{C}_3$.
	\end{proof}
	
	For the sake of completeness let us show that the class $\mathcal{C}_4$ also contains infinitely many Puiseux monoids that are infinitely molecular.
	
	\begin{proposition}
		There exists an infinite subclass of $\mathcal{C}_4$ consisting of infinitely molecular Puiseux monoids.
	\end{proposition}
	
	\begin{proof}
		Fix a prime $p$ such that $p \ge 5$, and then consider the Puiseux monoid $M_p := \big\langle M_{2/p} \cup \big\{ \frac{3}{2^n} \ \big{|} \ n \in \nn \big\} \big\rangle$, where $M_{2/p}$ is the Puiseux monoid $\langle (2/p)^n \mid n \in \nn_0 \rangle$. Clearly, $M_p$ is not atomic; indeed one can readily check that $3/2$ cannot be written as a sum of atoms. Hence $M_p$ belongs to the class $\mathcal{C}_4$.
		
		It is not hard to argue that $\mathcal{A}(M_p) = \mathcal{A}(M_{2/p})$. In addition, for every $n \in \nn$, \cite[Lemma~3.1]{CGG20} guarantees that the element $1 + (2/p)^n$ is a molecule of $M_{2/p}$ that is not an atom. This, together with the fact that $\mathcal{A}(M_p) = \mathcal{A}(M_{2/p})$, implies that $1 + (2/p)^n$ is a molecule of $M_p$ that is not an atom for every $n \in \nn$. As a consequence, $|\mathcal{M}(M_p) \setminus \mathcal{A}(M_p)| = \infty$. Finally, suppose that $M_p$ is isomorphic to $M_q$ for some $q \in \pp_{\ge 5}$. By \cite[Proposition~3.2]{fG18}, there exists $r \in \qq_{>0}$ such that $M_q = rM_p$. Since multiplication by $r$ is increasing, it must send $1 = \max \mathcal{A}(M_p)$ to $1 = \max \mathcal{A}(M_q)$ and it must send $2/p$ to $2/q$. Thus, $r = 1$, which implies that $q=p$. Hence $\mathcal{C}_4$ contains infinitely many non-isomorphic infinitely molecular Puiseux monoids.
	\end{proof}
	
	\bigskip
	
	
	\section*{Acknowledgments}
	
	The authors would like to thank Felix Gotti for many valuable suggestions during the preparation of this paper. The authors are also grateful to an anonymous referee whose feedback helped improve the final version of this paper.
	
	\bigskip
	


\begin{thebibliography}{20}
		
		\bibitem{AA10} D. D. Anderson and D.~F. Anderson: \emph{Factorization in integral domains IV}, Comm. Algebra {\bf 38} (2010) 4501--4513.
		
		\bibitem{AAZ90} D.~D. Anderson, D.~F. Anderson, and M. Zafrullah: \emph{Factorizations in integral domains}, J. Pure Appl. Algebra \textbf{69} (1990) 1--19.
		
		\bibitem{AAZ92} D.~D. Anderson, D.~F. Anderson, and M.~Zafrullah: \emph{Factorization in integral domains II}, J. Algebra {\bf 152} (1992) 78--93.
		
		\bibitem{AeA99} D.~F. Anderson and D.~N. El~Abidine: \emph{Factorization in integral domains III}, J. Pure Appl. Algebra {\bf 135} (1999) 107--127.
		
		\bibitem{AG16} A. Assi and P.~A. Garc\'ia-S\'anchez: \emph{Numerical Semigroups and Applications}. New York: Springer-Verlag, 2016.
		
		\bibitem{BG20} N. R. Baeth and F. Gotti: \emph{Factorization in upper triangular matrices over information semialgebras}, J. Algebra \textbf{562} (2020) 466--496.
		
		\bibitem{lC60} L. Carlitz: \emph{A characterization of algebraic number fields with class number two}, Proc. Amer. Math Soc., \textbf{11} (1960) 391--392.
		
		\bibitem{CGG20} S. T. Chapman, F. Gotti, and M. Gotti, \emph{Factorization invariants of Puiseux monoids generated by geometric sequences}, Comm. Algebra \textbf{48} (2020) 380--396.
		
		\bibitem{CGG19} S. T. Chapman, F. Gotti, and M. Gotti: \emph{How do elements really factor in $\zz[\sqrt{-5}]$?} In: Advances in Commutative Algebra (Eds. A. Badawi and J. Coykendall), pp. 171–195, Springer
		Trends in Mathematics, Birkh\"auser, Singapore, 2019.
		
		\bibitem{CGG20a} S. T. Chapman, F. Gotti, and M.  Gotti: \emph{When is a Puiseux monoid atomic?}, Amer. Math. Monthly (to appear). Available on arXiv: https://arxiv.org/pdf/1908.09227.pdf
		
		\bibitem{pC68} P. M. Cohn: \emph{Bezout rings and and their subrings}, Proc. Cambridge Philos. Soc. {\bf 64} (1968) 251--264.
		
		\bibitem{CG19} J. Coykendall and F. Gotti: \emph{On the atomicity of monoid algebras}, J. Algebra \textbf{539} (2019) 138--151.
		
		\bibitem{GR09} P.~A. Garc\'ia-S\'anchez and J.~C. Rosales: \emph{Numerical Semigroups}, Developments in Mathematics, 20, Springer-Verlag, New York, 2009.
		
		\bibitem{GR99} P.~A. Garc\'ia-S\'anchez and J.~C. Rosales: \emph{Numerical semigroups generated by intervals}, Pacific J. Math. {\bf 191} (1999), 75--83.
		
		\bibitem{GH06} A. Geroldinger and F.~Halter-Koch: \emph{Non-Unique Factorizations: Algebraic, Combinatorial and Analytic Theory}, Pure and Applied Mathematics, vol.~278, Chapman \& Hall/CRC, Boca Raton, 2006.
		
		\bibitem{fG20a} F. Gotti: \emph{Atomic and antimatter semigroup algebras with rational exponents}. Available on arXiv: https://arxiv.org/pdf/1801.06779v3.pdf
		
		\bibitem{fG19} F. Gotti: \emph{Increasing positive monoids of ordered fields are FF-monoids}, J. Algebra {\bf 518} (2019), 40--56.
		
		\bibitem{fG20} F. Gotti: \emph{Irreducibility and factorizations in monoid rings}. In: Numerical Semigroups (Eds. V. Barucci, S. T. Chapman, M. D'Anna, and R. Fröberg) pp. 129-139. Springer INdAM Series, Vol. 40, Switzerland, 2020.
		
		\bibitem{fG18} F. Gotti: \emph{Puiseux monoids and transfer homomorphisms}, J. Algebra \textbf{516} (2018) 95--114.
		
		\bibitem{fG20b} F. Gotti: \emph{The system of sets of lengths and the elasticity of submonoids of a finite-rank free commutative monoid}, J. Algebra Appl. \textbf{19} (2020) 2050137 (2020).
		
		\bibitem{GG20} F. Gotti and M. Gotti: \emph{On the molecules of numerical semigroups, Puiseux monoids, and monoid algebras}. In: Numerical Semigroups (Eds. V. Barucci, S. T. Chapman, M. D'Anna, and R. Fröberg) pp. 141--161. Springer INdAM Series, Vol. 40, Switzerland, 2020.
		
		\bibitem{aG74} A. Grams: \emph{Atomic rings and the ascending chain condition for principal ideals}. Math. Proc. Cambridge Philos. Soc. {\bf 75} (1974), 321--329.
		
		\bibitem{wN72} W. Narkiewicz: \emph{Numbers with unique factorization in an algebraic number field}, Acta Arith. {\bf 21} (1972) 313--322. 
		
		\bibitem{wN66} W. Narkiewicz: \emph{On natural numbers having unique factorization in a quadratic number field}, Acta Arith. {\bf 12} (1966) 1--22.
		
		\bibitem{wN66a} W. Narkiewicz: \emph{On natural numbers having unique factorization in a quadratic number field II}, Acta Arith. {\bf 13} (1967) 123--129.
		
	\end{thebibliography}
\end{document}